\documentclass[a4paper,11pt]{article}
\usepackage{amsmath, amsfonts, amscd, amssymb, amsthm, enumerate, xypic}

\def\id{\text{id}}

\newcommand{\Mat}{\operatorname{M}}

\newcommand{\GL}{\operatorname{GL}}
\newcommand{\Ker}{\operatorname{Ker}}

\newcommand{\im}{\operatorname{Im}}
\newcommand{\car}{\operatorname{car}}

\newcommand{\Sp}{\operatorname{Sp}}
\newcommand{\rk}{\operatorname{rk}}

\renewcommand{\setminus}{\smallsetminus}

\def\K{\mathbb{K}}

\def\C{\mathbb{C}}

\def\N{\mathbb{N}}

\theoremstyle{definition}
\newtheorem{Def}{Definition}
\newtheorem{Not}[Def]{Notation}

\theoremstyle{plain}
\newtheorem{theo}{Theorem}
\newtheorem{prop}[theo]{Proposition}
\newtheorem{cor}[theo]{Corollary}
\newtheorem{lemme}[theo]{Lemma}

\theoremstyle{plain}

\theoremstyle{remark}

\title{Sums of two triangularizable quadratic matrices over an arbitrary field}
\author{Cl\'ement de Seguins Pazzis\footnote{Professor of Mathematics at Lyc\'ee Priv\'e Sainte-Genevi\`eve, 2, rue
de l'\'Ecole des Postes, 78029 Versailles Cedex, FRANCE.}
\footnote{e-mail address: dsp.prof@gmail.com}}

\begin{document}

\thispagestyle{plain}

\maketitle

\begin{abstract}
Let $\K$ be an arbitrary field, and $a,b,c,d$ be elements of $\K$ such that
the polynomials $t^2-at-b$ and $t^2-ct-d$ are split in $\K[t]$. Given a square matrix $M \in \Mat_n(\K)$,
we give necessary and sufficient conditions for the existence of two matrices $A$ and $B$ such that $M=A+B$,
$A^2=a\,A+b\,I_n$ and $B^2=c\,B+d\,I_n$. Prior to this paper, such conditions were known in the case
$b=d=0$, $a \neq 0$ and $c \neq 0$ \cite{dSPidem2} and in the case $a=b=c=d=0$ \cite{Bothasquarezero}.
Here, we complete the study, which essentially amounts to determining when a matrix is the sum of an idempotent and a square-zero
matrix. This generalizes results of Wang \cite{Wanggeneral} to an arbitrary field, possibly of characteristic~$2$.
\end{abstract}



\section{Introduction}

\subsection{Basic notations and aims}

Let $\K$ be an arbitrary field, and $\overline{\K}$ an algebraic closure of it. We denote by $\car(\K)$ the characteristic of $\K$.
We denote by $\Mat_n(\K)$ the algebra of square matrices with $n$ rows and entries in $\K$, and by $I_n$
its identity matrix. Similarity of two square matrices $A$ and $B$ is denoted by $A \sim B$.
Given $M \in \Mat_n(\K)$, we denote by $\Sp(M)$ the set of eigenvalues
of $M$ in the field $\K$. We denote by $\N$ the set of non-negative integers, and by $\N^*$ the set of positive ones.

A matrix of $\Mat_n(\K)$ is called \textbf{quadratic} when it is annihilated by a polynomial of degree two.
More precisely, given a pair $(a,b)\in \K^2$, a matrix $A$ of $\Mat_n(\K)$ is called \textbf{$(a,b)$-quadratic} when
$A^2=a\,A+b\,I_n$. In particular, a matrix is $(1,0)$-quadratic if and only if it is idempotent, and
it is $(0,0)$-quadratic if and only if it is square-zero.

Let $(a,b,c,d)\in \K^4$. A matrix is called an \textbf{$(a,b,c,d)$-quadratic sum}
when it may be decomposed as the sum of an $(a,b)$-quadratic matrix and of a $(c,d)$-quadratic one.
Note that a matrix which is similar to an $(a,b,c,d)$-quadratic sum is an $(a,b,c,d)$-quadratic sum itself.
Our aim here is to give necessary and sufficient conditions for a matrix of $\Mat_n(\K)$ to be
an $(a,b,c,d)$-quadratic sum.
In \cite{Wanggeneral}, Wang has expressed such conditions in terms of rational canonical forms when $\K$
is the field of complex numbers, and his proof actually encompasses the more general case of an algebraically closed field
of characteristic not $2$. In our recent \cite{dSPidem2}, we have worked out the case $b=d=0$, $a \neq 0$ and $c \neq 0$, i.e., we have
determined when a matrix may be written as $a\,P+c\,Q$, where $P$ and $Q$ are idempotent matrices (this generalized
earlier results of Hartwig and Putcha \cite{HP}). In \cite{Bothasquarezero},
Botha has worked out the case $a=b=c=d=0$ for an arbitrary field, generalizing results of Wang and Wu \cite{WuWang};
as in \cite{dSPidem2}, fields of characteristic 2 yield somewhat different results than the others.

The purpose of this paper is to solve the remaining cases, assuming that the polynomials $t^2-a\,t-b$
and $t^2-c\,t-d$ are split over $\K$.

\vskip 3mm
The basic strategy is to reduce the situation to a more elementary one.
Assume, for the rest of the section, that $t^2-a\,t-b$
and $t^2-c\,t-d$ are split over $\K$, and let $\alpha$ be a root of $t^2-a\,t-b$ and $\beta$ be one of $t^2-c\,t-d$.
Then an $(a,b)$-quadratic matrix is a matrix of the form $\alpha\,I_n+ P$, where $P$ is $(a-2\alpha,0)$-quadratic.
We deduce that
a matrix of $\Mat_n(\K)$ is an $(a,b,c,d)$-quadratic sum if and only if it
splits as $(\alpha+\beta).I_n+M$, where $M$ is an $(a-2\alpha,0,c-2\beta,0)$-quadratic sum. \\
We are thus reduced to studying the case $b=d=0$. \\
In the case $b=d=0$ and $a \neq 0$, notice furthermore that an $(a,b,c,d)$-quadratic sum
is simply the product of $a$ with a $\bigl(1,0,\frac{c}{a},0\bigr)$-quadratic sum.
Therefore, the case $b=d=0$ is essentially reduced to three cases:
\begin{enumerate}[(i)]
\item $b=d=0$, $a\neq 0$ and $c \neq 0$;
\item $a=b=c=d=0$;
\item $a=1$ and $b=c=d=0$.
\end{enumerate}
Case (i) has been dealt with in \cite{dSPidem2}, and case (ii) more recently in \cite{Bothasquarezero}.
Therefore, only case (iii) remains to be studied in order to complete the case where both polynomials
$t^2-a t-b$ and $t^2-c t-d$ are split over $\K$. In other words, it remains to determine
which matrices may be decomposed as the sum of an idempotent and a square-zero matrix.
This has been done by Wang in \cite{Wanggeneral} for the case $\K=\C$. Our aim is to generalize his results.

\subsection{Main theorem}

\begin{Def}
Let $(u_n)_{n \geq 1}$ and $(v_n)_{n \geq 1}$ be two non-increasing sequences of non-negative integers.
Let $p>0$ be a positive integer.
We say that $(u_n)$ and $(v_n)$ are \textbf{$p$-intertwined} when
$$\forall n \geq 1, \; u_{n+p} \leq v_n \quad \text{and} \quad v_{n+p} \leq u_{n.}$$
\end{Def}

\begin{Not}
Given $A \in \Mat_n(\K)$, $\lambda \in \overline{\K}$ and $k \in \N^*$, we set
$$n_k(A,\lambda):=\dim \Ker (A-\lambda I_n)^k-\dim \Ker (A-\lambda I_n)^{k-1},$$
and
$$j_k(A,\lambda):=n_k(A,\lambda)-n_{k+1}(A,\lambda)$$
i.e., $n_k(A,\lambda)$ (respectively, $j_k(A,\lambda)$) is the number of blocks of size $k$ or more
(respectively, of size $k$) associated to the eigenvalue $\lambda$ in the Jordan reduction of $A$.
\end{Not}

Our main theorem follows.

\begin{theo}\label{maintheo}
Let $M \in \Mat_n(\K)$. The following conditions are equivalent:
\begin{enumerate}[(i)]
\item $M$ is a $(1,0,0,0)$-quadratic sum.
\item $\forall \lambda \in \overline{\K} \setminus \{0,1\}, \;\forall k \in \N^*, \;  j_k(M,\lambda)=j_k(M,1-\lambda)$,
the sequences $\bigl(n_k(M,0)\bigr)_{k \geq 1}$ and $\bigl(n_k(M,1)\bigr)_{k \geq 1}$ are $2$-intertwined,
and, if $\car(\K)\neq 2$, the Jordan blocks of $M$ for the eigenvalue $\frac{1}{2}$ are all even-sized.
\item There are matrices $A\in \Mat_p(\K)$ and $B \in \Mat_{n-p}(\K)$ such that $M \sim A \oplus B$, where
all the invariant factors of $A$ are polynomials of $t(t-1)$ and $A$ has no eigenvalue in $\{0,1\}$,
the matrix $B$ is triangularizable with $\Sp(B) \subset \{0,1\}$, and the
sequences $\bigl(n_k(B,0)\bigr)_{k \geq 1}$ and $\bigl(n_k(B,1)\bigr)_{k \geq 1}$ are $2$-intertwined.
\end{enumerate}
\end{theo}

\subsection{Structure of the proof}

The equivalence between conditions (ii) and (iii) of Theorem \ref{maintheo} is a straightforward consequence of the kernel decomposition theorem
and of Proposition 9 of \cite{dSPidem2}, which we restate:

\begin{prop}
Let $A \in \Mat_n(\K)$ and $\alpha \in \K$.
The following conditions are equivalent:
\begin{enumerate}[(i)]
\item The invariant factors of $A$ are polynomials of $t(t-\alpha)$.
\item For every $\lambda \in \overline{\K}$,
\begin{itemize}
\item if $\lambda \neq \alpha-\lambda$, then $\forall k \in \N^*, \; j_k(A,\lambda)=j_k(A,\alpha-\lambda)$;
\item if $\lambda=\alpha-\lambda$, then $\forall k \in \N, \; j_{2k+1}(A,\lambda)=0$.
\end{itemize}
\end{enumerate}
\end{prop}

\noindent The equivalence of (i) and (iii) is much more involving and takes up the rest of the paper:
\begin{itemize}
\item In Section \ref{redrecon}, we show that the equivalence (i) $\Leftrightarrow$ (iii)
needs to be proven only in the following elementary cases:
\begin{enumerate}[(a)]
\item $M$ has no eigenvalue in $\{0,1\}$ ;
\item $M$ is triangularizable and $\Sp(M) \subset \{0,1\}$.
\end{enumerate}
\item In Section \ref{pasdevpdans01}, we prove that (i) $\Leftrightarrow$ (iii) holds in case (a).
\item In Section \ref{vpdans01}, we prove that (i) $\Leftrightarrow$ (iii) holds in case (b).
\end{itemize}

\section{Reduction and reconstruction principles}\label{redrecon}

\subsection{A reconstruction principle}\label{recon}

Let $M_1$ and $M_2$ be two $(1,0,0,0)$-quadratic sums (respectively in $\Mat_n(\K)$ and $\Mat_p(\K)$).
Split up $M_1=A_1+B_1$ and $M_2=A_2+B_2$, where $A_1,A_2$ are idempotent and $B_1,B_2$ are square-zero.
Then $M_1 \oplus M_2=(A_1\oplus A_2)+(B_1 \oplus B_2)$, while $A_1\oplus A_2$ is idempotent and $B_1 \oplus B_2$ is square-zero.
Therefore $M_1 \oplus M_2$ is a $(1,0,0,0)$-quadratic sum.

\subsection{The basic lemma}

The following lemma is a key tool to analyze quadratic sums in general.

\begin{lemme}\label{corelemma}
Let $(a,b,c,d)\in \K^4$.
Let $A$ and $B$ be respectively an $(a,b)$-quadratic and a $(c,d)$-quadratic matrix of $\Mat_n(\K)$. \\
Then $A$ and $B$ both commute with $(A+B)\bigl((a+c)I_n-(A+B)\bigr)$.
\end{lemme}

\begin{proof}
Set $C:=(A+B)\bigl((a+c)I_n-(A+B)\bigr)$ and note that $C=(a+c)\,(A+B)-A^2-B^2-AB-BA=-(b+d)I_n+cA+aB-AB-BA$. \\
Therefore
$$AC-CA=a(AB-BA)-A^2B+BA^2=-b\,B+b\,B=0$$
and by symmetry $BC-CB=0$.
\end{proof}

\begin{cor}\label{commutecor}
Let $(A,B)\in \Mat_n(\K)^2$ such that $A^2=A$ and $B^2=0$. Then $A$ and $B$ both commute with $(A+B)(A+B-I_n)$.
\end{cor}

\subsection{Reduction to elementary cases}

Let $M \in \Mat_n(\K)$. The minimal polynomial $\mu$ of $M$ splits up as
$$\mu(t)=P(t)\,t^p\,(t-1)^q,$$
where $P(t)$ has no root in $\{0,1\}$ and $(p,q)\in \N^2$.
Let $M_1$ (respectively, $M_2$) be a matrix associated to the endomorphism $X \mapsto MX$ on the vector space $\Ker P(M)$
(respectively, on the vector space $\Ker M^p(M-I_n)^q$).
By the kernel decomposition theorem, one has
$$M \sim M_1 \oplus M_2,$$
while $P(M_1)=0$ and $t^p(t-1)^q$ annihilates $M_2$.
If implication (iii) $\Rightarrow$ (i) holds for $M_1$ and $M_2$, then the reconstruction principle of Section \ref{recon}
shows that it also holds for $M$.

\vskip 2mm
Conversely, assume that $M=A+B$ for a pair $(A,B)\in \Mat_n(\K)^2$ with $A^2=A$ and $B^2=0$.
By Corollary \ref{commutecor}, $A$ and $B$ both commute with $M(M-I_n)$, and hence they stabilize
the subspaces $\im \bigl(M(M-I_n)\bigr)^n$ and $\Ker \bigl(M(M-I_n)\bigr)^n$
in the Fitting decomposition of $M(M-I_n)$. Using an adapted basis of $\K^n$ for this decomposition, we find $P \in \GL_n(\K)$, an integer
$p \geq 0$, matrices $A_1,B_1$ in $\Mat_p(\K)$ and matrices $A_2,B_2$ in $\Mat_{n-p}(\K)$ such that
$$A=P\bigl(A_1\oplus A_2\bigr) P^{-1} \quad \text{and} \quad B=P\bigl(B_1\oplus B_2\bigr) P^{-1},$$
the matrices $M_1:=A_1+B_1$ and $M_2:=A_2+B_2$ being both $(1,0,0,0)$-quadratic sums,
with $M_1(M_1-I_p)$ non-singular and $M_2(M_2-I_{n-p})$ nilpotent. In other words,
$M_1$ has no eigenvalue in $\{0,1\}$ and $M_2$ is triangularizable with $\Sp(M_2) \subset \{0,1\}$.
If implication (i) $\Rightarrow$ (iii) holds for both $M_1$ and $M_2$, then it clearly holds for $M$.

\vskip 2mm
\noindent We conclude that equivalence (i) $\Leftrightarrow$ (iii) needs to be proven only in the following special cases:
\begin{enumerate}[(a)]
\item $M$ has no eigenvalue in $\{0,1\}$;
\item $M$ is triangularizable with $\Sp(M) \subset \{0,1\}$.
\end{enumerate}

\section{The case $M$ has no eigenvalue in $\{0,1\}$}\label{pasdevpdans01}

\subsection{A lemma on companion matrices}

\begin{Not}
Given a monic polynomial $P=t^n-a_{n-1}t^{n-1}-\dots-a_1t-a_0 \in \K[t]$, we denote its \emph{companion matrix} by
$$C(P):=\begin{bmatrix}
0 & 0 & \dots  & & 0 & a_0 \\
1 & 0 & & & 0 & a_1 \\
0 & 1 & 0 & \dots & 0 & a_2 \\
 &  & \ddots & \ddots & & \vdots\\
\vdots &  & &  1 & 0 & a_{n-2} \\
0 & \dots & \dots & 0 & 1 & a_{n-1}
\end{bmatrix} \in \Mat_n(\K).$$
\end{Not}

\begin{Not}
For $E \in \Mat_p(\K)$, we set
$$U_E:=\begin{bmatrix}
I_p & E \\
I_p & 0_p
\end{bmatrix} \in \Mat_{2p}(\K).$$
\end{Not}

We start with two easy lemmas on the matrices of type $U_E$.

\begin{lemme}\label{simU}
Given two similar matrices $E$ and $E'$ of $\Mat_p(\K)$, the matrices
$U_E$ and $U_{E'}$ are similar.
\end{lemme}

\begin{proof}
Choosing $R \in \GL_p(\K)$ such that $E'=RER^{-1}$, a straightforward computation shows that
$$U_{E'}=(R \oplus R)\, U_E\, (R \oplus R)^{-1}.$$
\end{proof}

Conjugating by a well-chosen permutation matrix, the following result is straightforward:

\begin{lemme}\label{sumU}
Given square matrices $A$ and $B$, one has
$U_{A \oplus B} \sim U_A \oplus U_B$.
\end{lemme}

We now examine the case $E$ is a companion matrix.
The following lemma generalizes Lemma 14 of \cite{dSPidem2} and is the key to equivalence (i) $\Leftrightarrow$ (iii) in Theorem \ref{maintheo}
for a matrix with no eigenvalue in $\{0,1\}$:

\begin{lemme}\label{companionlemma}
Let $(\alpha,\beta)\in \K^2$. Let $P(t)$ be a monic polynomial of degree $n$.
Then
$$\begin{bmatrix}
\alpha\, I_n & C(P) \\
I_n & \beta\, I_n
\end{bmatrix} \sim C\bigl(P\bigl((t-\alpha)(t-\beta)\bigr)\bigr).$$
\end{lemme}

Lemma \ref{companionlemma} was stated and proved in \cite{dSPidem2} with the extra condition that $\alpha \neq 0$ and $\beta \neq 0$,
but an inspection of the proof shows that this condition is unnecessary.

\begin{cor}\label{corcompanion}
Let $P \in \K[t]$ be a monic polynomial. Then the companion matrix $C\bigl(P(t(t-1))\bigr)$
is a $(1,0,0,0)$-quadratic sum.
\end{cor}

\begin{proof}
Indeed, Lemma \ref{companionlemma} shows, with $n:=\deg P$, that
$$C\bigl(P(t(t-1))\bigr) \sim A+B \quad \text{with} \quad A=\begin{bmatrix}
I_n & 0_n \\
I_n & 0_n
\end{bmatrix} \quad \text{and} \quad B=\begin{bmatrix}
0_n & C(P) \\
0_n & 0_n
\end{bmatrix}.$$
Obviously, $A^2=A$ and $B^2=0$, and hence $C\bigl(P(t(t-1))\bigr)$ is the sum of an idempotent and a square-zero matrix.
\end{proof}

\subsection{Application to $(1,0,0,0)$-quadratic sums}\label{equivcomp}

Let $M \in \Mat_n(\K)$.

\begin{itemize}
\item Assume that each invariant factor of $M$ is a polynomial of $t(t-1)$.
Then we may find monic polynomials $P_1,\dots,P_p$ such that
$$M \sim C\bigl(P_1(t(t-1))\bigr) \oplus \cdots \oplus C\bigl(P_p(t(t-1))\bigr).$$
Using Corollary \ref{corcompanion} and the reconstruction principle of Section \ref{recon},
we deduce that $M$ is a $(1,0,0,0)$-quadratic sum.

\item Conversely, assume that $M=A+B$ for some pair $(A,B) \in \Mat_n(\K)^2$ such that $A^2=A$ and $B^2=0$.
Assume furthermore that $M$ has no eigenvalue in $\{0,1\}$. This last assumption yields
$$\Ker A \cap \Ker B=\Ker(A-I_n) \cap \Ker B=\{0\}.$$
Therefore
$$\dim \Ker A \leq n-\dim \Ker B=\rk B \quad \text{and} \quad \dim \Ker(A-I_n) \leq \rk B.$$
Adding these inequalities yields $n \leq 2\rk B$. However $2\rk B \leq \rk B+\dim \Ker B=n$
since $\im B \subset \Ker B$. It follows that
$$\dim \Ker A=\dim \Ker(A-I_n)=\dim \Ker B=\rk B=\frac{n}{2}$$
and hence
$$\K^n=\Ker A\oplus \Ker B.$$
Set now $p:=\frac{n}{2}\cdot$ Using a basis of $\K^{2p}$ which is adapted to the decomposition $E=\Ker B \oplus \Ker A$,
we find $P \in \GL_n(\K)$ and matrices $C,D$ in $\Mat_p(\K)$ such that
$$A=P\,\begin{bmatrix}
I_p & 0_p \\
C & 0_p
\end{bmatrix} \,P^{-1} \quad \text{and} \quad
B=P\,\begin{bmatrix}
0_p & D \\
0_p & 0_p
\end{bmatrix} \,P^{-1}.$$
Using $\Ker (A-I_n)\cap \Ker B=\{0\}$, we find that $C$ is non-singular.
Setting $Q:=\begin{bmatrix}
I_p & 0 \\
0 & C
\end{bmatrix}$, we finally find some $D' \in \Mat_p(\K)$ such that
$$M=(PQ)\,\begin{bmatrix}
I_p & D' \\
I_p & 0_p
\end{bmatrix} \,(PQ)^{-1} \, \sim \, U_{D'}.$$
The rational canonical form of $D'$ yields monic polynomials $P_1,\dots,P_q$
such that $D' \sim C(P_1) \oplus \cdots \oplus C(P_q)$ and $P_k$ divides $P_{k+1}$ for every $k \in \{1,\dots,q-1\}$.
By Lemmas \ref{simU} and \ref{sumU}, this yields
$$M \sim U_{C(P_1)} \oplus \cdots \oplus U_{C(P_q)}.$$
Using Corollary \ref{corcompanion}, it follows that
$$M \sim C\bigl(P_1(t(t-1))\bigr) \oplus \cdots \oplus C\bigl(P_q(t(t-1))\bigr).$$
Finally, $P_k(t(t-1))$ divides $P_{k+1}(t(t-1))$ for every $k \in \{1,\dots,q-1\}$,
and hence $P_1(t(t-1)),\dots,P_q(t(t-1))$ are the invariant factors of $M$.
Since $M$ has no eigenvalue in $\{0,1\}$, we conclude that $M$ satisfies condition (iii) in Theorem \ref{maintheo}.
\end{itemize}

\noindent We conclude that equivalence (i) $\Leftrightarrow$ (iii) of Theorem \ref{maintheo} holds for any square matrix
with no eigenvalue in $\{0,1\}$.

\section{The case $M$ is triangularizable with eigenvalues in $\{0,1\}$}\label{vpdans01}

\subsection{A review of Wang's results}

In \cite[Lemma 2.3]{Wanggeneral}, Wang proved the following characterization of
pairs of nilpotent matrices $(M,N)$ for which the sequences $(n_k(M,0))_{k \geq 1}$ and $(n_k(N,0))_{k \geq 1}$
are $p$-intertwined (generalizing a famous theorem of Flanders \cite{Flanders}).

\begin{theo}[Wang]\label{Wangtheorem}
Let $p \in \N^*$ and $(M,N) \in \Mat_r(\K) \times \Mat_s(\K)$ be a pair of
nilpotent matrices. The following conditions are equivalent:
\begin{enumerate}[(i)]
\item The sequences $(n_k(M,0))_{k \geq 1}$ and $(n_k(N,0))_{k \geq 1}$ are $p$-intertwined.
\item There is a pair $(X,Y) \in \Mat_{r,s}(\K) \times \Mat_{s,r}(\K)$ such that $M^p=XY$,
$N^p=YX$, $MX=XN$ and $YM=NY$.
\end{enumerate}
\end{theo}

Wang only considered the field of complex numbers but an inspection of his proof reveals
that it holds for an arbitrary field.

In \cite{Wanggeneral}, implication (i) $\Rightarrow$ (ii) of Theorem \ref{Wangtheorem}
is used, with $p=2$, to obtain the following result:

\begin{prop}\label{castrigdirect}
Let $M \in \Mat_n(\K)$ be a triangularizable matrix with eigenvalues in $\{0,1\}$
and assume that the sequences $(n_k(M,0))_{k \geq 1}$ and $(n_k(M,1))_{k \geq 1}$ are $2$-intertwined.
Then $M$ is a $(1,0,0,0)$-quadratic sum.
\end{prop}

Again, Wang's proof \cite[Lemma 2.2, ``Sufficiency" paragraph]{Wanggeneral} holds for an arbitrary field
and we shall not reproduce it. We deduce that implication (iii) $\Rightarrow$ (i) in Theorem \ref{maintheo}
holds when $M$ is triangularizable with eigenvalues in $\{0,1\}$.

\subsection{A necessary condition for being a $(1,0,0,0)$-quadratic sum}\label{lastsection}

Here, we prove the converse of Proposition \ref{castrigdirect}:

\begin{prop}\label{castrigrecip}
Let $M \in \Mat_n(\K)$ be a triangularizable matrix with eigenvalues in $\{0,1\}$.
Assume that $M$ is a $(1,0,0,0)$-quadratic sum. Then the sequences $(n_k(M,0))_{k \geq 1}$ and $(n_k(M,1))_{k \geq 1}$ are $2$-intertwined.
\end{prop}

Proving this will complete our proof of Theorem \ref{maintheo}.

In \cite{Wanggeneral}, Wang proved Proposition \ref{castrigrecip} in the special case $\K=\C$.
An inspection shows that his proof works for an arbitrary field of characteristic not $2$,
but fails for a field of characteristic $2$ (due to Wang's systematic use of the division by $2$). Our aim is to give a proof that works regardless of the characteristic of $\K$.
In order to do this, we will reduce the situation to the one where no Jordan block of $M$ has a size greater than $3$
(in other words $M^3(M-I_n)^3=0$). Let us start by considering that special case:

\begin{lemme}\label{elelemme}
Let $M \in \Mat_n(\K)$ be a $(1,0,0,0)$-quadratic sum such that $M^3(M-I_n)^3=0$.
Then $n_3(M,0)\leq n_1(M,1)$ and $n_3(M,1) \leq n_1(M,0)$.
\end{lemme}

\begin{proof}
We lose no generality in assuming that
$$M=\begin{bmatrix}
I_p+N & 0 \\
0 & N'
\end{bmatrix},$$
where $p+q=n$, $(N,N')\in \Mat_p(\K) \times \Mat_q(\K)$, and $N^3=0$ and $(N')^3=0$.
With the same block sizes, we may find some $B=\begin{bmatrix}
B_1 & B_3 \\
B_2 & B_4
\end{bmatrix} \in \Mat_n(\K)$ such that $B^2=0$ and $(M-B)^2=M-B$.
By Corollary \ref{commutecor}, $B$ commutes with $M(M-I_n)=\begin{bmatrix}
N^2+N & 0 \\
0 & (N')^2-N'
\end{bmatrix}$. It follows that $B_1$ commutes with $N+N^2$, whilst $B_4$ commutes with $N'-(N')^2$. \\
However $N=(N+N^2)-(N+N^2)^2$ and $N'=(N'-(N')^2)+(N'-(N')^2)^2$. Therefore
$B_1$ commutes with $N$, and $B_4$ commutes with $N'$. \\
Next, the identities $(M-B)^2=M-B$ and $B^2=0$ yield:
$$M^2-MB-BM=M-B.$$
We deduce:
$$N'B_2+B_2N=0 \quad ; \quad NB_3+B_3N'=0,$$
$$N^2+N=NB_1+B_1N+B_1=(2N+I_p)B_1 \quad \text{and} \quad (N')^2-N'=(2N'-I_q)B_4.$$
Therefore
$$B_1=(I_p+2N)^{-1}(N+N^2)=(I_p-2N+4N^2)(N+N^2)=N-N^2$$
and
$$B_4=(I_q-2N')^{-1}(N'-(N')^2)=(I_q+2N'+4(N')^2)(N'-(N')^2)=N'+(N')^2.$$
Using this, we compute
$$B^2=\begin{bmatrix}
N^2+B_3B_2 & ? \\
 ?          & (N')^2+B_2B_3
\end{bmatrix}.$$
Since $B^2=0$, we deduce that
$$N^2=(-B_3)B_2 \quad \text{and} \quad  (-N')^2=B_2(-B_3).$$
Recalling that
$$(-N')B_2=B_2N \quad \text{and} \quad N(-B_3)=(-B_3)(-N'),$$
Theorem \ref{Wangtheorem} yields $n_3(N,0) \leq n_1(-N',0)$ and $n_3(-N',0) \leq n_1(N,0)$,
i.e., $n_3(M,1) \leq n_1(M,0)$ and $n_3(M,0) \leq n_1(M,1)$.
\end{proof}

We finish by deducing the general case from the above special one:

\begin{proof}[Proof of Proposition \ref{castrigrecip}]
We think in terms of endomorphisms of the space $\K^n$.
Let $u$ be an endomorphism of $\K^n$ such that $u^n(u-\id)^n=0$,
and assume that there is an idempotent endomorphism $a$ and a square-zero endomorphism $b$ such that $u=a+b$. \\
By Corollary \ref{commutecor}, $E_k:=\Ker \bigl(u^k(u-\id)^k\bigr)$ is stabilized by $a$ and $b$ for every $k \in \N$.
Let $k \in \N$. Then $a$, $b$ and $u$ induce endomorphisms $a'$, $b'$ and $u'$ of
$E_{k+3}/E_k$, with $(a')^2=a'$, $(b')^2=0$, and $(u')^3(u'-\id)^3=0$ (as $u^3(u-\id)^3$ maps $E_{k+3}$ into $E_k$).
Applying Lemma \ref{elelemme} to $u'$, we find that
$n_3(u',1) \leq n_1(u',0)$ and $n_3(u',0) \leq n_1(u',1)$.
In order to conclude, it suffices to note that
$$\forall i \in \{1,2,3\}, \; n_i(u',0)=n_{k+i}(u,0) \quad \text{and} \quad n_i(u',1)=n_{k+i}(u,1).$$
Note indeed, using the kernel decomposition theorem, that the characteristic subspace of $u'$ for the eigenvalue $0$ is
$(\Ker u^{k+3} \oplus \Ker(u-\id)^k)/(\Ker u^k \oplus \Ker(u-\id)^k)$, and hence the nilpotent part of
$u'$ is similar to the endomorphism $v : x \mapsto u(x)$ of $\Ker u^{k+3}/\Ker u^k$.
However $\Ker v^i=\Ker u^{k+i}/\Ker u^k$ for every $i \in \{0,1,2,3\}$. Therefore
\begin{multline*}
n_i(u',0)=n_i(v,0)=(\dim \Ker u^{k+i}-\dim \Ker u^k)-(\dim \Ker u^{k+i-1}-\dim \Ker u^k)\\
=n_{k+i}(u,0)
\end{multline*}
for every $i \in \{1,2,3\}$. In the same way, one proves that $n_i(u',1)=n_{k+i}(u,1)$ for every $i \in \{1,2,3\}$. \\
The special cases $i=1$ and $i=3$ yield $n_{k+3}(u,1) \leq n_{k+1}(u,0)$ and $n_{k+3}(u,0) \leq n_{k+1}(u,1)$.
\end{proof}

\noindent This completes our proof of Theorem \ref{maintheo}.

\section{Addendum : a simplified proof of a result on linear combinations of idempotent matrices}\label{appendix}

In this last section, we wish to show how the strategy of Section \ref{lastsection}
may be adapted so as to yield a simplified proof of the following result of \cite{dSPidem2}:

\begin{prop}
Let $\alpha,\beta$ be distinct elements of $\K \setminus \{0\}$.
Let $M \in \Mat_n(\K)$ be an $(\alpha,0,\beta,0)$-quadratic sum such that $(M-\alpha I_n)^n(M-\beta I_n)^n=0$.
Then the sequences $(n_k(M,\alpha))_{k \geq 1}$ and $(n_k(M,\beta))_{k \geq 1}$ are $1$-intertwined.
\end{prop}

\begin{proof}
As in the proof of Proposition \ref{castrigrecip}, one
can use the commutation with $(M-\alpha I_n)(M-\beta I_n)=M(M-(\alpha+\beta)I_n)+\alpha \beta\,I_n$
(see Lemma \ref{corelemma})
to reduce the situation to the one where $(M-\alpha I_n)^2(M-\beta I_n)^2=0$.
In that case, we lose no generality in assuming that
$$M=\bigl(\alpha I_p+N\bigr) \oplus \bigl(\beta I_q+N'\bigr),$$
where $p+q=n$, $N \in \Mat_p(\K)$ and $N'\in \Mat_q(\K)$ satisfy $N^2=0$ and $(N')^2=0$.
Note that
$$(M-\alpha I_n)(M-\beta I_n)=(\alpha-\beta)\,(N \oplus (-N')).$$
Let then $A$ and $B$ be idempotent matrices such that $M=\alpha\,A+\beta\,B$.
Split
$$A=\begin{bmatrix}
A_1 & A_3 \\
A_2 & A_4
\end{bmatrix},$$
where $A_1,A_2,A_3,A_4$ are respectively $p \times p$, $q \times p$, $p \times q$ and $q \times q$ matrices. \\
By Lemma \ref{corelemma}, $A$ commutes with $(M-\alpha I_n)(M-\beta I_n)$; as $\alpha \neq \beta$, we deduce that $A_1$ commutes with $N$. \\
On the other hand, the identity
$(M-\alpha A)^2=\beta(M-\alpha A)$ yields:
$$\alpha(\alpha+\beta)A=\alpha\,(AM+MA)+\beta\,M-M^2.$$
Evaluating the upper-left blocks on both sides and using the commutation $A_1N=NA_1$, we deduce:
$$\alpha(\alpha+\beta)A_1=2\alpha\,(\alpha\,I_n+N)A_1+\beta\,(\alpha I_n+N)-(\alpha I_n+N)^2$$
and hence
$$\alpha\bigl((\beta-\alpha)I_n-2N\bigr)A_1=\alpha(\beta-\alpha)I_n+(\beta-2\alpha)N.$$
As $\alpha(\beta-\alpha) \neq 0$ and $N^2=0$, we deduce that
$$A_1=\Bigl(I_n+\frac{\beta-2\alpha}{\alpha (\beta-\alpha)}N\Bigr)\Bigl(I_n-\frac{2}{\beta-\alpha}N\Bigr)^{-1}
=I_n+\frac{\beta}{\alpha(\beta-\alpha)}\,N,$$
and it follows that the upper-left block of $B$ is $\frac{1}{\beta}\bigl(\alpha I_n+N-\alpha A_1\bigr)=\frac{\alpha}{\beta(\alpha-\beta)}\,N$.
By symmetry, one has $A_4=\frac{\beta}{\alpha(\beta-\alpha)}\,N'$.
We deduce that
$$A-A^2=\begin{bmatrix}
\frac{\beta}{\alpha(\alpha-\beta)}\,N-A_3A_2 & ? \\
? & \frac{\beta}{\alpha(\beta-\alpha)}\,N'-A_2A_3
\end{bmatrix}.$$
Setting $X:=\alpha(\alpha-\beta)A_3$ and $Y:=\frac{1}{\beta} A_2$, we find:
$$N=XY \quad \text{and} \quad -N'=YX.$$
The main theorem of \cite{Flanders} (or Theorem \ref{Wangtheorem} for $p=1$, noting that $NX=XYX=X(-N')$ and $YN=YXY=(-N')Y$)
then shows that the sequences $(n_k(N,0))_{k \geq 1}$ and $(n_k(-N',0))_{k \geq 1}$
are $1$-intertwined, i.e., the sequences $(n_k(M,\alpha))_{k \geq 1}$ and $(n_k(M,\beta))_{k \geq 1}$
are $1$-intertwined.
\end{proof}

\end{document}